\documentclass[12pt, a4paper]{amsart}
\usepackage{amsmath,amssymb,amsthm,booktabs,longtable,enumerate}
\usepackage{comment}
\usepackage{graphicx}
\usepackage[all]{xy}
\usepackage{caption}
\usepackage{color}
\theoremstyle{plain}
\newtheorem{theorem}{Theorem}[section]

\newtheorem{proposition}[theorem]{Proposition}

\newtheorem{fact}[theorem]{Fact}

\newtheorem{remark}[theorem]{Remark}

\newtheorem{question}[theorem]{Question}

\newcommand{\Int}{\mathop{\mathrm{Int}}\nolimits}

\newcommand{\DIM}{\mathop{\mathrm{DIM}}\nolimits}

\newcommand{\Hom}{\mathop{\mathrm{Hom}}\nolimits}

\newcommand{\Irr}{\mathop{\mathrm{Irr}}\nolimits}

\newcommand{\R}{\mathbb{R}}
\newcommand{\C}{\mathbb{C}}

\newcommand{\Z}{\mathbb{Z}}

\makeatletter
    
    \@addtoreset{equation}{section}
  \makeatother
\usepackage{enumerate}%
\makeatother
\begin{document}

\title[Real minimal nilpotent orbits and symmetric pairs]{Some remarks on real minimal nilpotent orbits and symmetric pairs}
\author{Takayuki Okuda}
\thanks{The third author is supported by JSPS Grants-in-Aid for Scientific Research JP20K03589, JP20K14310,  JP22H0112 and JP24K06714}

\address[T.~Okuda]{%
	Graduate School of Advanced Science and Engineering, Hiroshima University, 
    1-3-1 Kagamiyama, Higashi-Hiroshima City, Hiroshima, 739-8526, Japan.
        }
\email{okudatak@hiroshima-u.ac.jp}

\maketitle



\begin{abstract}
For a non-compact simple Lie algebra $\mathfrak{g}$ over $\R$,
we denote by $\mathcal{O}^{\C}_{\min,\mathfrak{g}}$ 
the unique complex nilpotent orbit in $\mathfrak{g} \otimes_\R \C$ containing all minimal real nilpotent orbits in $\mathfrak{g}$.
In this paper, 
we give a complete classification 
of symmetric pairs $(\mathfrak{g},\mathfrak{h})$ such that 
$\mathcal{O}^{\C}_{\min,\mathfrak{g}} \cap \mathfrak{g}^d = \emptyset$,
where $\mathfrak{g}^d$ denotes the dual Lie algebra of $(\mathfrak{g},\mathfrak{h})$.
Furthermore, 
for symmetric pairs $(G,H)$ with real simple Lie group $G$, 
we apply our classification to 
theorems given by T.~Kobayashi  [J.~Lie Theory (2023)], 
and study bounded multiplicity properties 
of restrictions on $H$ of infinite-dimensional irreducible $G$-representations 
with minimum Gelfand--Kirillov dimension.
\end{abstract}


\section{Introduction}\label{section:Intro}

Let $\mathfrak{g}$ be a non-compact simple Lie algebra over $\R$.
We write $\mathcal{N}(\mathfrak{g})$ for the nilpotent cone in $\mathfrak{g}$, and $\mathcal{N}(\mathfrak{g})/G$ for the set of real nilpotent (adjoint) orbits in $\mathfrak{g}$.
A non-zero real nilpotent orbit $\mathcal{O}$ in $\mathfrak{g}$ with minimum dimension is called \emph{minimal}.
Let us put $\mathcal{N}_{\min}(\mathfrak{g}) \subset \mathcal{N}(\mathfrak{g})$ the union of all minimal real nilpotent orbits in $\mathfrak{g}$.
The complexification of $\mathfrak{g}$ is denoted by $\mathfrak{g}_\C := \mathfrak{g} \otimes_\R \C = \mathfrak{g} + \sqrt{-1} \mathfrak{g}$.
Then as proved in \cite[Theorems 1.1 and 1.3]{Okuda2015smallest}, there exists uniquely a complex nilpotent (adjoint) orbit $\mathcal{O}^{\C}_{\min,\mathfrak{g}}$ in $\mathfrak{g}_\C$ containing $\mathcal{N}_{\min}(\mathfrak{g})$.
We denote by $m(\mathfrak{g})$ half of the complex dimension of $\mathcal{O}^{\C}_{\min,\mathfrak{g}}$.
Note that the orbit $\mathcal{O}^\C_{\min,\mathfrak{g}}$ is minimal as a complex nilpotent orbit in most cases
 but not for some real simple $\mathfrak{g}$
(see Section \ref{section:minimalreal} for more details).

Let us also fix an involutive automorphism $\sigma$ on $\mathfrak{g}$,
and put $\mathfrak{h} := \{ X \in \mathfrak{g} \mid \sigma(X) = X \}$ and $\mathfrak{q} := \{ X \in \mathfrak{g} \mid \sigma(X) = -X \}$.
Then $(\mathfrak{g},\mathfrak{h})$ is a symmetric pair.
Take a Cartan involution $\theta$ on $\mathfrak{g}$ commuting with $\sigma$, 
and write $\mathfrak{g} = \mathfrak{k} + \mathfrak{p}$ for the Cartan decomposition.
Then as in \cite{OshimaSekiguchi84}, 
the dual Lie algebra $\mathfrak{g}^d$ of $(\mathfrak{g},\mathfrak{h})$ is defined by 
\begin{align*}
    \mathfrak{g}^d := (\mathfrak{k} \cap \mathfrak{h}) + \sqrt{-1}(\mathfrak{k} \cap \mathfrak{q}) + \sqrt{-1} (\mathfrak{p} \cap \mathfrak{h}) + (\mathfrak{p} \cap \mathfrak{q}).
\end{align*}
Note that $\mathfrak{g}^d$ is a real form of $\mathfrak{g}_\C$, 
and up to inner-automorphisms on $\mathfrak{g}_\C$, 
the real form $\mathfrak{g}^d$ of $\mathfrak{g}_\C$ depends only on $(\mathfrak{g},\sigma)$, but not on $\theta$. 

In this paper, we address the following question:
\begin{question}\label{question:Omingmeetsgd}
Whether $\mathcal{O}^{\C}_{\min,\mathfrak{g}}$ meets $\mathfrak{g}^d$ or not?
\end{question}
We find that the orbit $\mathcal{O}^\C_{\min,\mathfrak{g}}$ meets $\mathfrak{g}^d$ in most cases. 
More precisely, our main result gives a complete classification of symmetric pairs $(\mathfrak{g},\mathfrak{h})$ with 
$\mathcal{O}^{\C}_{\min,\mathfrak{g}} \cap \mathfrak{g}^d = \emptyset$
(see Theorems \ref{theorem:main2} and \ref{theorem:main3} in Section \ref{section:classification} for the classification).

Our work is motivated by branching problems of ``small'' infinite-dimensional representations of simple Lie groups as below:  
let $G$ be a non-compact connected simple Lie group with Lie algebra $\mathfrak{g}$.
The set of irreducible objects in the category of smooth admissible representations of $G$ of finite length with moderate growth is denoted by $\Irr(G)$.
Then for each infinite-dimensional irreducible object $\Pi \in \Irr(G)$, 
the Gelfand-Kirillov dimension $\DIM(\Pi)$ of $\Pi$ satisfies 
\[
m(\mathfrak{g}) \leq \DIM(\Pi).
\]

By combining our classification with T.~Kobayashi's results in \cite{KobayashiMRMR2022, Kobayashi2023BoundedSymmetric}, 
we have the following theorem, mentioned in \cite[Remark 5.7]{Kobayashi2023BoundedSymmetric}:

\begin{theorem}\label{theorem:bmp}
Let $G$ be a non-compact connected simple Lie group, and $\Pi \in \Irr(G)$ with $m(\mathfrak{g}) = \DIM(\Pi)$.
Then for any symmetric pair $(G,H)$, 
the restriction $\Pi|_H$ has the bounded multiplicity property
\[
\sup_{\pi \in \Irr(H)} [\Pi|_{H}:\pi] < \infty, 
\]
where $[\Pi|_H:\pi]$ denotes the multiplicity of $\pi$ in $\Pi|_{H}$ given by $\dim_\C \Hom_{H}(\Pi|_H,\pi)$.
\end{theorem}

Question \ref{question:Omingmeetsgd} is also related to the following topological problem: 
let $G/H^a$ be the symmetric space corresponding to 
the associated symmetric pair of $(G,H)$, in the sense of \cite{Berger_classification}, and 
let $SL_2(\R)$ or $PSL_2(\R)$ be a subgroup of $G$ 
corresponding to a real minimal nilpotent orbit 
by the Jacobson--Morozov theorem (see Section \ref{section:properSL2} for the details).
In Theorem \ref{theorem:SL2properGHa}, we show that $\mathcal{O}^\C_{\min,\mathfrak{g}} \cap \mathfrak{g}^d = \emptyset$ if and only if 
the action of the subgroup on $G/H^a$ is proper. 

\section{Preliminaries for minimal real nilpotent orbits}\label{section:minimalreal}

Let $\mathfrak{g}$ be a non-compact simple Lie algebra over $\R$.
In this section, we recall some facts and give some observations for minimal real nilpotent orbits in $\mathfrak{g}$.

As in Section \ref{section:Intro}, 
we denote by $\mathcal{N}(\mathfrak{g})$ and $\mathcal{N}(\mathfrak{g})/G$ 
the nilpotent cone and the set of nilpotent (adjoint) orbits in $\mathfrak{g}$.
With respect to the closure ordering, 
$\mathcal{N}(\mathfrak{g})/G$ is a partially ordered set, 
and the zero orbit is the minimum in $\mathcal{N}(\mathfrak{g})/G$.
A non-zero real nilpotent orbit $\mathcal{O}$ in $\mathfrak{g}$ is called minimal
if its dimension is minimum, 
or equivalently, it is minimal in the partially ordered set $(\mathcal{N}(\mathfrak{g})/G) \setminus \{ 0 \}$
(see \cite[Theorem 1.3]{Okuda2015smallest}).
We write $\mathcal{N}_{\min}(\mathfrak{g})$ for the union of all minimal real nilpotent orbits in $\mathfrak{g}$.

For the number of minimal real nilpotent orbits, 
the following is known:

\begin{fact}[See e.g.~{\cite{Okuda2015smallest}}]
Let $\mathfrak{g} = \mathfrak{k} + \mathfrak{p}$ be a Cartan decomposition of $\mathfrak{g}$.
\begin{enumerate}
    \item If $\mathfrak{g}$ is absolutely-simple and $(\mathfrak{g},\mathfrak{k})$ is of Hermitian-type, then $\mathfrak{g}$ admits two minimal real nilpotent orbits $\mathcal{O}_{\min,1}$ and $\mathcal{O}_{\min,2}$ with $\mathcal{O}_{\min,2} = -\mathcal{O}_{\min,1}$.   
    \item If $\mathfrak{g}$ is absolutely-simple but $(\mathfrak{g},\mathfrak{k})$ is not of Hermitian-type, then the minimal real nilpotent orbit $\mathcal{O}_{\min}$ in $\mathfrak{g}$ is unique.
    \item If $\mathfrak{g}$ is not absolutely-simple, 
    then the minimal real nilpotent orbit $\mathcal{O}_{\min}$ in $\mathfrak{g}$ is unique.
    Furthermore, in this situation, $\mathfrak{g}$ admits two complex structures $\{ \pm J \}$, 
    and $\mathcal{O}_{\min}$ can be considered as the minimal complex nilpotent orbit in the complex simple Lie algebra $(\mathfrak{g},\pm J)$. 
\end{enumerate}
\end{fact}

We denote by $\mathfrak{g}_\C = \mathfrak{g} \otimes_\R \C = \mathfrak{g} + \sqrt{-1} \mathfrak{g}$ the complexification of $\mathfrak{g}$.
Then there uniquely exists a complex nilpotent adjoint orbit $\mathcal{O}^{\C}_{\min,\mathfrak{g}}$ in $\mathfrak{g}_\C$ with $\mathcal{N}_{\min}(\mathfrak{g}) \subset \mathcal{O}^{\C}_{\min,\mathfrak{g}}$ (see \cite[Theorems 1.1 and 1.3]{Okuda2015smallest}).
We define $m(\mathfrak{g}) \in \Z_{\geq 1}$ as half of the complex dimension of $\mathcal{O}^{\C}_{\min,\mathfrak{g}}$.
Note that for each real minimal nilpotent orbit $\mathcal{O}_{\min}$ in $\mathfrak{g}$, 
we have $m(\mathfrak{g}) = \frac{1}{2} \dim_\R \mathcal{O}_{\min}$.

We also recall that for each 
complex simple Lie algebra $\mathfrak{l}_\C$, 
the minimal complex nilpotent (adjoint) orbit $\mathcal{O}^{\C}_{\min}$ in $\mathfrak{l}_\C$ is defined as the unique non-zero complex nilpotent orbit in $\mathfrak{l}_\C$ which is contained in the closure of any non-zero complex nilpotent orbit in $\mathfrak{l}_\C$.
Note that $\dim_\C \mathcal{O}^{\C}_{\min} \leq \dim_{\C} \mathcal{O}^\C$ holds for any non-zero complex nilpotent orbit $\mathcal{O}^\C$
and the equality holds only if $\mathcal{O}^\C = \mathcal{O}^{\C}_{\min}$.
Throughout this paper, 
we define $n(\mathfrak{l}_\C)$ as half of the complex dimension of $\mathcal{O}^{\C}_{\min}$.
Table \ref{table:nl_C} gives the formula of $n(\mathfrak{l}_\C)$ (see e.g.~\cite[Lemma 4.3.5 and Chapter 8.4]{CollingwoodMcGovern}).
Furthermore, for a complex semisimple Lie algebra $\mathfrak{l}'_\C$, 
we also put $n(\mathfrak{l}'_\C) := \sum_i n(\mathfrak{l}^i_\C)$, where 
$\mathfrak{l}'_\C = \bigoplus_i \mathfrak{l}^i_\C$ denotes the complex simple ideal decomposition of $\mathfrak{l}'_\C$.

    \begin{table}[!h]
    \caption{List of $n(\mathfrak{l}_\C)$ for complex simple $\mathfrak{l}_\C$}
    \begin{tabular}{c|ccccccccc}
        \hline
           $\mathfrak{l}_\C$ & $\mathfrak{sl}_{n}(\C)$ & $\mathfrak{so}_n(\C)$ & $\mathfrak{sp}_n(\C)$ & $\mathfrak{e}_{6}^\C$ & $\mathfrak{e}_{7}^\C$ & $\mathfrak{e}_{8}^\C$ & $\mathfrak{f}_{4}^\C$ & $\mathfrak{g}_{2}^\C$ \\
        \hline
        \hline
           $n(\mathfrak{l}_\C)$ & $n-1$ & $n-3$ & $n$ & $11$ & $17$ & $29$ & $8$ & $3$ \\
        \hline
    \end{tabular}
    \label{table:nl_C}
    \end{table}

If our real Lie algebra $\mathfrak{g}$ is absolutely-simple, 
then its complexification $\mathfrak{g}_\C$ is complex simple and the closure of $\mathcal{O}^{\C}_{\min,\mathfrak{g}}$ contains the minimal complex nilpotent orbit $\mathcal{O}^{\C}_{\min}$ in $\mathfrak{g}_\C$,
and thus $n(\mathfrak{g}_\C) \leq m(\mathfrak{g})$.
Furthermore, 
if $\mathfrak{g}$ is not absolutely simple, 
for the complex structures $\{ \pm J \}$ on $\mathfrak{g}$, we have complex simple Lie algebras $(\mathfrak{g}, J)$ and $(\mathfrak{g}, -J)$. 
The complexification $\mathfrak{g}_\C$ of $\mathfrak{g}$ can be identified with $(\mathfrak{g},J) \oplus (\mathfrak{g},-J)$ by 
\[
\mathfrak{g}_\C \rightarrow (\mathfrak{g},J) \oplus (\mathfrak{g},-J), ~ X + \sqrt{-1} Y \mapsto (X + J Y,X - J Y).
\]
Let us denote by $\mathcal{O}_{\min}$ 
the unique minimal nilpotent orbit in $\mathfrak{g}$.
Then $\mathcal{O}^{\C}_{\min,\mathfrak{g}}$ can be identified with $\mathcal{O}_{\min} \times \mathcal{O}_{\min}$.
In particular, we obtain 
\[
n(\mathfrak{g}_\C) = n((\mathfrak{g},J)) + n((\mathfrak{g},-J)) = m(\mathfrak{g}).
\]

The situation $n(\mathfrak{g}_\C) < m(\mathfrak{g})$ occurs only if $\mathfrak{g}$ is absolutely-simple, and classified as below:

\begin{fact}[See also {\cite{Brylinski1998realminimalnilp}, \cite[Proposition 30]{KobayashiMRMR2022}, \cite[Corollary 5.9]{KobayashiOshimaY2015symmpair},  \cite[Proposition 4.1]{Okuda2015smallest}}]\label{fact:smallestnilpcomp}
The following conditions on a non-compact absolutely-simple Lie algebra $\mathfrak{g}$ are equivalent:
\begin{enumerate}
    \item $n(\mathfrak{g}_\C) < m(\mathfrak{g})$.
    \item $\mathcal{O}^{\C}_{\min} \cap \mathfrak{g} = \emptyset$.
    \item $\mathcal{O}^{\C}_{\min} \neq \mathcal{O}^{\C}_{\min,\mathfrak{g}}$.
    \item $\mathfrak{g}$ is isomorphic to one of the Lie algebras in Table \ref{table:mingneqmin}.
\end{enumerate}
Furthermore, if $\mathfrak{g}$ satisfies the equivalent conditions above, then the symmetric pair $(\mathfrak{g},\mathfrak{k})$ is not of Hermitian-type, 
 and $\mathfrak{g}$ has a unique minimal real nilpotent orbit.
\end{fact}

    \begin{table}[h]
    \centering
    \caption{List of $\mathfrak{g}$ with $n(\mathfrak{g}_\C) < m(\mathfrak{g})$}
        \label{table:mingneqmin}
    \begin{tabular}{c|c|c|c}
        \hline
           $\mathfrak{g}$ & $m(\mathfrak{g})$ & $(\mathfrak{g}_\C,\mathfrak{k}_\C)$ & $n(\mathfrak{g}_\C)$ \\
        \hline
        \hline
            $\mathfrak{su}^*_{2n}$ $(n \geq 2)$& $4n-4$ & $(\mathfrak{sl}_{2n}(\C),\mathfrak{sp}_n(\C))$ & $2n-1$ \\
            $\mathfrak{so}_{n-1,1}$ $(n \geq 5)$ & $n-2$ & $(\mathfrak{so}_{n}(\C),\mathfrak{so}_{n-1}(\C))$ & $n-3$ \\
            $\mathfrak{sp}_{m,n}$ ($m,n \geq 1$) & $2(m+n)-1$ & $(\mathfrak{sp}_{m+n}(\C),\mathfrak{sp}_m(\C) \oplus \mathfrak{sp}_{n}(\C))$ & $m+n$ \\
            $\mathfrak{e}_{6(-26)}$ & $16$ & $(\mathfrak{e}_{6}^\C,\mathfrak{f}_4^\C)$ & $11$ \\
            $\mathfrak{f}_{4(-20)}$ & $11$ & $(\mathfrak{f}_{4}^\C,\mathfrak{so}_9(\C))$ & $8$ \\
        \hline
    \end{tabular}
    \end{table}

The following observation will be applied in the next section:

\begin{proposition}\label{proposition:key}
Suppose that $n(\mathfrak{g}_\C) < m(\mathfrak{g})$.
Let $\mathfrak{g}'$ be a non-compact real form of $\mathfrak{g}_\C$.
Then 
$\mathcal{O}^{\C}_{\min,\mathfrak{g}} \cap \mathfrak{g}' = \emptyset$ 
if and only if 
$(\mathfrak{g},\mathfrak{g}')$ is isomorphic to $(\mathfrak{su}^*_{2n},\mathfrak{su}_{2n-1,1})$ ($n \geq 2$) or $(\mathfrak{so}_{2n-1,1},\mathfrak{so}^*_{2n})$ ($n \geq 4$).
\end{proposition}

\begin{proof}[Proof of Proposition \ref{proposition:key}]
As stated in \cite[Section 4.1]{Okuda2015smallest} or \cite[Proposition 7.8 and Theorem 7.10]{Okuda13}, 
a complex nilpotent orbit $\mathcal{O}^\C$ in $\mathfrak{g}_\C$ meets a fixed real form $\mathfrak{g}'$ if and only if 
the weighted Dynkin diagram of $\mathcal{O}^\C$ 
satisfies the matching condition to the Satake diagram $S_{\mathfrak{g}'}$ of $\mathfrak{g}'$.
For the cases where $n(\mathfrak{g}_\C) < m(\mathfrak{g})$, or equivalently, $\mathcal{O}^{\C}_{\min,\mathfrak{g}} \neq \mathcal{O}^{\C}_{\min}$, 
the list of weighted Dynkin diagrams of $\mathcal{O}^{\C}_{\min,\mathfrak{g}}$ can be found in \cite[Table 1]{Okuda2015smallest}.
By comparing it with the list of Satake diagrams of absolutely simple Lie algebras (see e.g.~\cite{Araki62} or \cite[Chapter X]{Helgason2001bookDLSS}), 
one can easily obtain the claim in Proposition \ref{proposition:key}.
\end{proof}

\section{Classifications}\label{section:classification}

In this section, we classify symmetric pairs $(\mathfrak{g},\mathfrak{h})$ with $\mathcal{O}^{\C}_{\min,\mathfrak{g}} \cap \mathfrak{g}^d = \emptyset$.

\subsection{In the case where $\mathfrak{g}$ is absolutely-simple}

Let $\mathfrak{g}$ be a non-compact absolutely simple Lie algebra over $\R$, and 
$(\mathfrak{g},\mathfrak{h})$ a symmetric pair.

One of the main results of this paper is given below:

\begin{theorem}\label{theorem:main1}
Suppose that $n(\mathfrak{g}_\C) < m(\mathfrak{g})$. Then $\mathcal{O}^{\C}_{\min,\mathfrak{g}} \cap \mathfrak{g}^d \neq \emptyset$.
\end{theorem}

\begin{proof}[Proof of Theorem \ref{theorem:main1}]
Let $(\mathfrak{g},\mathfrak{h})$ be a symmetric pair with absolutely simple $\mathfrak{g}$.
By Proposition \ref{proposition:key}, 
we only need to check that 
the pair $(\mathfrak{g},\mathfrak{g}^d)$ is not isomorphic to 
$(\mathfrak{su}^*_{2n},\mathfrak{su}_{2n-1,1})$ ($n \geq 2$) nor $(\mathfrak{so}_{2n-1,1},\mathfrak{so}^*_{2n})$ ($n \geq 4$).
By the tables in \cite[Section 1]{OshimaSekiguchi84}, 
one can see the following:
\begin{itemize}
    \item If $\mathfrak{g} \simeq \mathfrak{su}^*_{2n}$ $(n \geq 2)$, then $\mathfrak{g}^d$ is isomorphic to $\mathfrak{su}^*_{2n}$, $\mathfrak{su}_{2n-2j,2j}$, $\mathfrak{sl}_{2n}(\R)$ or $\mathfrak{su}_{n,n}$.
    \item If $\mathfrak{g}^d \simeq \mathfrak{so}^*_{2n}$ ($n \geq 4$),
    then $\mathfrak{g}$ is isomorphic to $\mathfrak{so}^*_{2n}$, $\mathfrak{so}_{n,n}$ or $\mathfrak{so}_{2n-2j,2j}$.
\end{itemize}
These complete the proof.
\end{proof}

By combining Theorem \ref{theorem:main1} with Fact \ref{fact:smallestnilpcomp} and tables of dual pairs of symmetric pairs in \cite{OshimaSekiguchi84}, 
we have the following classification result:

\begin{theorem}\label{theorem:main2}
Let $(\mathfrak{g},\mathfrak{h})$ be a symmetric pair such that $\mathfrak{g}$ is absolutely-simple.
Then the five conditions on $(\mathfrak{g},\mathfrak{h})$ below are equivalent:
\begin{enumerate}
    \item $\mathcal{O}^{\C}_{\min,\mathfrak{g}} \cap \mathfrak{g}^d = \emptyset$.
    \item $m(\mathfrak{g}) = n(\mathfrak{g}_\C) < m(\mathfrak{g}^d)$.
    \item $\mathcal{O}^{\C}_{\min} = \mathcal{O}^{\C}_{\min,\mathfrak{g}}$ but $\mathcal{O}^{\C}_{\min} \neq \mathcal{O}^{\C}_{\min,\mathfrak{g}^d}$.
    \item The simple Lie algebra $\mathfrak{g}^d$ can be found in Table \ref{table:mingneqmin} but $\mathfrak{g}$ can not.
    \item The symmetric pair $(\mathfrak{g},\mathfrak{h})$ is isomorphic to one of the symmetric pairs in Table \ref{table:gabssimple}. 
\end{enumerate}
\end{theorem}

    \begin{table}[h]
    \centering
    \caption{List of $(\mathfrak{g},\mathfrak{h})$ with absolutely simple $\mathfrak{g}$ such that $\mathcal{O}^{\C}_{\min,\mathfrak{g}} \cap \mathfrak{g}^d = \emptyset$}
        \label{table:gabssimple}
    \begin{tabular}{c|c|c}
        \hline
           $(\mathfrak{g},\mathfrak{h})$ &  $\mathfrak{g}^d$ & $\mathfrak{h}^a = \mathfrak{g} \cap \mathfrak{g}^d$ \\
        \hline
        \hline
        $(\mathfrak{sl}_{2n}(\R),\mathfrak{sp}_{n}(\R))$ $(n \geq 2)$ & $\mathfrak{su}^*_{2n}$ & $\mathfrak{sl}_{n}(\C) \oplus \mathfrak{so}_2$ \\
        $(\mathfrak{su}_{2n-2j,2j},\mathfrak{sp}_{n-j,j})$ $(n \geq 2)$ & $\mathfrak{su}^*_{2n}$ & $\mathfrak{sp}_{n-j,j}$ \\
        $(\mathfrak{su}_{n,n},\mathfrak{sp}_{n}(\R))$ $(n \geq 2)$ & $\mathfrak{su}^*_{2n}$ & $\mathfrak{so}^*_{2n}$  \\
        $(\mathfrak{so}_{m,n},\mathfrak{so}_{m-1,n})$ ($m,n \geq 2$ \text{ with } $(m,n) \neq (2,2)$) & $\mathfrak{so}_{m+n-1,1}$ & $\mathfrak{so}_{1,n} \oplus \mathfrak{so}_{m-1}$ \\
        $(\mathfrak{sp}_{n}(\R),\mathfrak{sp}_{n-j}(\R) \oplus \mathfrak{sp}_j(\R))$ ($n \geq 2$, $n-1 \geq j \geq 1$)& $\mathfrak{sp}_{n-j,j}$ & $\mathfrak{su}_{n-j,j} \oplus \mathfrak{so}_2$ \\
        $(\mathfrak{sp}_{2n}(\R),\mathfrak{sp}_{n}(\C))$ $(n \geq 1)$ & $\mathfrak{sp}_{n,n}$ & $\mathfrak{sp}_{n}(\C)$  \\
        $(\mathfrak{e}_{6(6)},\mathfrak{f}_{4(4)})$ & $\mathfrak{e}_{6(-26)}$ & $\mathfrak{su}^*_6 \oplus \mathfrak{su}_2$ \\
        $(\mathfrak{e}_{6(2)},\mathfrak{f}_{4(4)})$ & $\mathfrak{e}_{6(-26)}$ & $\mathfrak{sp}_{3,1}$ \\
        $(\mathfrak{e}_{6(-14)},\mathfrak{f}_{4(-20)})$ & $\mathfrak{e}_{6(-26)}$ & $\mathfrak{f}_{4(-20)}$ \\
        $(\mathfrak{f}_{4(4)},\mathfrak{so}_{5,4})$ & $\mathfrak{f}_{4(-20)}$ & $\mathfrak{sp}_{2,1} \oplus \mathfrak{su}_2$ \\  
        \hline
    \end{tabular}
    \end{table}

\begin{remark}
In Table \ref{table:gabssimple}, 
the Riemannian symmetric pair $(\mathfrak{g},\mathfrak{k})$ is of Hermitian type
if and only if $\mathfrak{g}$ is isomorphic to 
$\mathfrak{su}_{2n-2j,2j}$,  $\mathfrak{su}_{n,n}$, $\mathfrak{so}_{m,n}$ with $\min\{ m,n \} = 2$, $\mathfrak{sp}_n(\R)$, $\mathfrak{sp}_{2n}(\R)$ or $\mathfrak{e}_{6(-14)}$. 
\end{remark}

\subsection{In the case where $\mathfrak{g}$ is not absolutely simple}

Let $\mathfrak{g}$ be a non-compact simple Lie algebra over $\R$ and 
$(\mathfrak{g},\mathfrak{h})$ a symmetric pair.
In this subsection, we assume that $\mathfrak{g}$ is not absolutely simple.
Then $\mathfrak{g}$ admits complex structures $\{ \pm J \}$ (see Section \ref{section:minimalreal}).
For the involution $\sigma$ on $\mathfrak{g}$ defining $\mathfrak{h}$, 
either one of the following two situations occurs:
\begin{itemize}
    \item $\sigma \circ J = (-J) \circ \sigma$.
    \item $\sigma \circ J = J \circ \sigma$.
\end{itemize}
We fix a Cartan involution $\theta$ on $\mathfrak{g}$ commuting with $\sigma$.
Then $\theta \circ J = (-J) \circ \theta$,
and the Cartan decomposition of $\mathfrak{g}$ corresponding to $\theta$ can be written as $\mathfrak{g} = \mathfrak{u} + J \mathfrak{u}$ 
by a compact real form $\mathfrak{u}$ of $(\mathfrak{g},\pm J)$.

Let us consider the case where $\sigma \circ J = (-J) \circ \sigma$, 
that is, $\sigma$ is anti-$\C$-linear on $(\mathfrak{g},\pm J)$.
Then the subalgebra $\mathfrak{h}$ is a real form of $(\mathfrak{g},\pm J)$, 
and the real form $\mathfrak{g}^d$ of $\mathfrak{g}_\C = (\mathfrak{g},J) \oplus (\mathfrak{g},-J)$ can be written as 
\[
\{ (X,\theta \sigma X) \in (\mathfrak{g},J) \oplus (\mathfrak{g},-J) \mid X \in \mathfrak{g} \}.
\]
In particular, the simple Lie algebra $\mathfrak{g}^d$ over $\R$ is isomorphic to $\mathfrak{g}$. 

Next, we shall suppose that $\sigma \circ J = J \circ \sigma$, that is, $\sigma$ is $\C$-linear on $(\mathfrak{g},\pm J)$.
Then $\mathfrak{h}$ is $\pm J$-stable.
In this situation, 
$\mathfrak{u} \cap \mathfrak{h}$ gives a real form of $(\mathfrak{h},\pm J)$.
We define the real form 
\[
\mathfrak{g}^d_0 := (\mathfrak{u} \cap \mathfrak{h}) + J (\mathfrak{u} \cap \mathfrak{q})
\]
of $(\mathfrak{g},\pm J)$.
Then the real form $\mathfrak{g}^d$ of $\mathfrak{g}_\C = (\mathfrak{g},J) \oplus (\mathfrak{g},-J)$ can be identified with $\mathfrak{g}^d_0 \oplus \mathfrak{g}^d_0$,
and $\mathfrak{h}$ can be understood as the complexification of the maximal compact subalgebra $\mathfrak{u} \cap \mathfrak{h}$ of $\mathfrak{g}^d_0$.

The theorem below gives a classification of $(\mathfrak{g},\mathfrak{h})$ with $\mathcal{O}^{\C}_{\min,\mathfrak{g}} \cap \mathfrak{g}^d = \emptyset$ in the setting above:

\begin{theorem}\label{theorem:main3}
Let $(\mathfrak{g},\mathfrak{h})$ be a symmetric pair with simple $\mathfrak{g}$ admitting two complex structures $\{ \pm J \}$.
\begin{enumerate}[(1)]
    \item \label{item:main3:antiholo} Suppose that $\sigma \circ J = (-J) \circ \sigma$. Then $m(\mathfrak{g}) = n(\mathfrak{g}_\C) = m(\mathfrak{g}^d)$, and $\mathcal{O}^{\C}_{\min,\mathfrak{g}} \cap \mathfrak{g}^d \neq \emptyset$.
    \item \label{item:main3:holo} Suppose that $\sigma \circ J = J \circ \sigma$.
        Then the following six conditions on $(\mathfrak{g},\mathfrak{h})$ are equivalent:
\begin{enumerate}[(i)]
    \item \label{item:main3hol:minCgd} $\mathcal{O}^{\C}_{\min,\mathfrak{g}} \cap \mathfrak{g}^d = \emptyset$. 
    \item \label{item:main3hol:nmm} $m(\mathfrak{g}) = n(\mathfrak{g}_\C) < m(\mathfrak{g}^d)$ where we put $m(\mathfrak{g}^d) = 2m(\mathfrak{g}^d_0)$.
    \item \label{item:main3hol:mind0} $\mathcal{O}_{\min} \cap \mathfrak{g}^d_0 = \emptyset$.
    \item \label{item:main3hol:nm0} $n((\mathfrak{g},\pm J)) < m(\mathfrak{g}^d_0)$.
    \item \label{item:main3hol:gd0list} $\mathfrak{g}^d_0$ is isomorphic to one of the Lie algebras in Table \ref{table:mingneqmin}.
    \item \label{item:main3hol:ghlist} The complex symmetric pair $(\mathfrak{g},\mathfrak{h})$ is isomorphic to one of the symmetric pairs in Table \ref{table:gcomplex}.
\end{enumerate}
\end{enumerate}
\end{theorem}

    \begin{table}[h]
    \centering
    \caption{List of $(\mathfrak{g},\mathfrak{h})$ with non-absolutely-simple $\mathfrak{g}$ such that $\mathcal{O}^{\C}_{\min,\mathfrak{g}} \cap \mathfrak{g}^d = \emptyset$}
    \label{table:gcomplex}
    \begin{tabular}{c|c}
        \hline
           $(\mathfrak{g},\mathfrak{h})$ & $\mathfrak{g}^d_0 = \mathfrak{h}^a$ \\
        \hline
        \hline
        $(\mathfrak{sl}_{2n}(\C),\mathfrak{sp}_n(\C))$ $(n \geq 2)$ & $\mathfrak{su}^{*}_{2n}$ \\
        $(\mathfrak{so}_n(\C),\mathfrak{so}_{n-1}(\C))$ ($n \geq 5$) & $\mathfrak{so}_{n-1,1}$ \\
        $(\mathfrak{sp}_{m+n}(\C),\mathfrak{sp}_{m}(\C) \oplus \mathfrak{sp}_{n}(\C))$ $(n+m \geq 2, n,m \geq 1)$. &$\mathfrak{sp}_{m,n}$ \\
        $(\mathfrak{e}^\C_{6},\mathfrak{f}^{\C}_4)$ & $\mathfrak{e}_{6(-26)}$ \\        
        $(\mathfrak{f}^\C_{4},\mathfrak{so}_{9}(\C))$ & $\mathfrak{f}_{4(-20)}$ \\
        \hline
    \end{tabular}
    \end{table}

\begin{proof}[Proof of Theorem \ref{theorem:main3}]
As we mentioned in Section \ref{section:minimalreal},  
under the identification of $\mathfrak{g}_\C$ with $(\mathfrak{g},J) \oplus (\mathfrak{g},-J)$, 
the orbit $\mathcal{O}^{\C}_{\min,\mathfrak{g}}$ coincides with 
$\mathcal{O}_{\min} \times \mathcal{O}_{\min}$,
and $n(\mathfrak{g}_\C) = m(\mathfrak{g})$.

First, let us suppose that $\sigma \circ J = (-J) \circ \sigma$.
Then the real form $\mathfrak{g}^d$ of $\mathfrak{g}_\C$ is simple but not absolutely-simple since $\mathfrak{g}^d$ is isomorphic to $\mathfrak{g}$.
Thus by the arguments in Section \ref{section:minimalreal} again, 
we have $n(\mathfrak{g}_\C) = m(\mathfrak{g}^d)$ 
and $\mathcal{O}^{\C}_{\min,\mathfrak{g}^d} = \mathcal{O}_{\min} \times \mathcal{O}_{\min}$.
Thus we have $m(\mathfrak{g}) = n(\mathfrak{g}_\C) = m(\mathfrak{g}^d)$, and $\mathcal{O}^{\C}_{\min,\mathfrak{g}} \cap \mathfrak{g}^d \neq \emptyset$.
This completes the proof of the claim \eqref{item:main3:antiholo}.

Next, suppose that $\sigma \circ J = J \circ \sigma$.
The equivalence among 
\eqref{item:main3hol:minCgd}, 
\eqref{item:main3hol:mind0},
\eqref{item:main3hol:nm0} and 
\eqref{item:main3hol:gd0list}
comes from 
$\mathfrak{g}^d = \mathfrak{g}^d_0 \oplus \mathfrak{g}^{d}_0$,  $\mathcal{O}^{\C}_{\min,\mathfrak{g}} = \mathcal{O}_{\min} \times \mathcal{O}_{\min}$ and Fact \ref{fact:smallestnilpcomp}.
The equivalence \eqref{item:main3hol:nm0} $\Leftrightarrow$ \eqref{item:main3hol:nmm} is followed by $n(\mathfrak{g}_\C) = n(\mathfrak{g},J) + n(\mathfrak{g},-J)$, $n(\mathfrak{g},J) = n(\mathfrak{g},-J)$ and $m(\mathfrak{g}^d) = 2 m(\mathfrak{g}^d_0)$.
Finally, we obtain the equivalence \eqref{item:main3hol:gd0list} $\Leftrightarrow$
\eqref{item:main3hol:ghlist} 
by the observation that $(\mathfrak{g},\pm J)$ is isomorphic to the complexification of $\mathfrak{g}^d_0$, and $\mathfrak{h}$ is the complexification of the maximal compact subalgebra $\mathfrak{u} \cap \mathfrak{g}^d_0$ of $\mathfrak{g}^d_0$
in $(\mathfrak{g},\pm J)$.
The proof of the claim \eqref{item:main3:holo} has been completed.
\end{proof}

\begin{remark}\label{remark:cplxified}
For each symmetric pair $(\mathfrak{g},\mathfrak{h})$, 
the dual Lie algebra of the complexified symmetric pair $(\mathfrak{g} \otimes_\R \C,\mathfrak{h} \otimes_\R \C)$ can be written as $\mathfrak{g}^d \oplus \mathfrak{g}^d$.
Therefore, if $(\mathfrak{g},\mathfrak{h})$ satisfies the conditions in Theorem \ref{theorem:main2}, 
then $(\mathfrak{g} \otimes_\R \C,\mathfrak{h} \otimes_\R \C)$  satisfies the conditions in Theorem \ref{theorem:main3}.    
Note that the converse claim does not hold.
In fact, suppose $\mathfrak{g}$ is absolutely simple 
and both $m(\mathfrak{g})$ and $m(\mathfrak{g}^d)$ are greater than $n(\mathfrak{g}_\C)$, 
then $(\mathfrak{g} \otimes_\R \C,\mathfrak{h} \otimes_\R \C)$ 
satisfies the conditions in Theorem \ref{theorem:main3}
but the conditions in Theorem \ref{theorem:main2} do not hold for $(\mathfrak{g},\mathfrak{h})$.
The symmetric pair $(\mathfrak{g},\mathfrak{h}) = (\mathfrak{su}^*_{2n},\mathfrak{sp}_{2n-j,j})$ with $\mathfrak{g}^d = \mathfrak{su}^*_{2n}$ is one of examples.
\end{remark}

\section{Bounded Multiplicity branching for symmetric pairs}\label{section:BMB}

In this section, as an application of Theorem \ref{theorem:main1},
we give a proof of Theorem \ref{theorem:bmp} stated in Section \ref{section:Intro}.

Let $G$ be a non-compact simple Lie group with Lie algebra $\mathfrak{g}$.
For $\Pi \in \Irr(G)$ and a reductive subgroup $H$ of $G$, 
as in \cite{KobayashiMRMR2022, Kobayashi2023BoundedSymmetric}, 
we say that the restriction $\Pi|_H$ has the bounded multiplicity property if 
\[
\sup_{\pi \in \Irr(H)} [\Pi|_{H}:\pi] < \infty.
\]

T. Kobayashi \cite{Kobayashi2023BoundedSymmetric} gives the following theorem for bounded multiplicity properties of restrictions of irreducible $G$-representation with ``small'' Gelfand--Kirillov dimensions for symmetric pairs:

\begin{fact}[{\cite[Theorem 3]{KobayashiMRMR2022}, \cite[Theorems 5.5 (1) and 6.1]{Kobayashi2023BoundedSymmetric}}]\label{fact:bmpKob}
Let $\Pi \in \Irr(G)$.
\begin{enumerate}[(1)]
    \item \label{item:bmpkob1} Suppose that $n(\mathfrak{g}_\C) = \DIM(\Pi)$.
Then for any symmetric pair $(G,H)$, 
the restriction $\Pi|_H$ has the bounded multiplicity property.
    \item \label{item:bmpkob2} Suppose that $m(\mathfrak{g}) = \DIM(\Pi)$.
Take a symmetric pair $(G,H)$ such that the corresponding symmetric pair $(\mathfrak{g},\mathfrak{h})$ of Lie algebras satisfies the condition ``$\sigma \mu = -\mu$'' (see below for the details).
Then the restriction $\Pi|_H$ has the bounded multiplicity property.
\end{enumerate}
\end{fact}

Let us explain the condition ``$\sigma \mu = -\mu$'' for a symmetric pair $(\mathfrak{g},\mathfrak{h})$ with simple $\mathfrak{g}$. 
As in Section \ref{section:Intro}, we denote by $\sigma$ the involution on $\mathfrak{g}$ defining $\mathfrak{h}$, and put $\mathfrak{q} := \{ X \in \mathfrak{g} \mid \sigma(X) = -X \}$.
Fix a Cartan involution $\theta$ on $\mathfrak{g}$ commuting with $\sigma$, 
and write $\mathfrak{g} = \mathfrak{k} + \mathfrak{p}$ for the Cartan decomposition.
We take a maximal abelian subspace $\mathfrak{a}^{-\sigma}$ of $\mathfrak{p} \cap \mathfrak{q}$, 
and extend it to a maximal abelian subspace $\mathfrak{a}$ of $\mathfrak{p}$.
Note that $\mathfrak{a}$ is $\sigma$-stable.
The dual of the vector space $\mathfrak{a}$ is denoted by $\mathfrak{a}^*$, 
and we use the same symbol $\sigma$ for the involution on $\mathfrak{a}^{*}$ induced by $\sigma$ on $\mathfrak{a}$.
Then the root system $\Sigma(\mathfrak{g},\mathfrak{a})$ is known to be $\sigma$-stable.
Let us also fix a lexicographic ordering on $\mathfrak{a}^{-\sigma}$, and extend it to that on $\mathfrak{a}$.
The positive system 
of the root system $\Sigma(\mathfrak{g},\mathfrak{a})$ corresponding to the fixed ordering is denoted by $\Sigma^+(\mathfrak{g},\mathfrak{a})$.
We write $\mu$ for the highest element in $\Sigma^{+}(\mathfrak{g},\mathfrak{a})$.
Then the vectors $\sigma \mu$ and $-\mu$ in $\mathfrak{a}^*$ are both defined.
Note that the condition ``$\sigma \mu = -\mu$'' depends only on $(\mathfrak{g},\mathfrak{h})$.

Theorem \ref{theorem:bmp}
 can be obtained as a corollary to Kobayashi's theorem (Fact \ref{fact:bmpKob}), Theorem \ref{theorem:main1} and the following proposition:

\begin{proposition}\label{proposition:sigmamucondition_nilp_condition}
Let $(\mathfrak{g},\mathfrak{h})$ be a symmetric pair with simple $\mathfrak{g}$.
Then the condition ``$\sigma \mu = -\mu$'' holds if and only if $\mathcal{O}^{\C}_{\min,\mathfrak{g}} \cap \mathfrak{g}^d \neq \emptyset$.
\end{proposition}

\begin{proof}[Proof of Proposition \ref{proposition:sigmamucondition_nilp_condition}]
Let us denote by the coroot $\mu^\vee \in \mathfrak{a}$ of the highest root $\mu \in \Sigma(\mathfrak{g},\mathfrak{a})$.
We denote by $\mathcal{O}_{\mu^\vee}$ the real adjoint orbit in $\mathfrak{g}$ through $\mu^\vee$, 
and by $\mathcal{O}_{\mu^\vee}^\C$ the complex adjoint orbit in $\mathfrak{g}_\C$ through $\mu^\vee$.
We shall prove that the following six conditions on $(\mathfrak{g},\mathfrak{h})$ are equivalent:
\begin{enumerate}
    \item \label{item::munilp:nilp} $\mathcal{O}^\C_{\min,\mathfrak{g}}$ meets $\mathfrak{g}^d$.
    \item \label{item::munilp:hyp} $\mathcal{O}^\C_{\mu^\vee}$ meets $\mathfrak{g}^d$.
    \item \label{item::munilp:realhyp} $\mathcal{O}_{\mu^\vee}$ meets $\mathfrak{g} \cap \mathfrak{g}^d = \mathfrak{h}^a = (\mathfrak{k} \cap \mathfrak{h}) + (\mathfrak{p} \cap \mathfrak{q})$.
    \item \label{item::munilp:a} $\mathcal{O}_{\mu^\vee}$ meets $\mathfrak{a}^{-\sigma} \cap \mathfrak{a}_+$ at one point, where $\mathfrak{a}_+$ denotes the closed Weyl chamber corresponding to the positive system $\Sigma^+(\mathfrak{g},\mathfrak{a})$.
    \item \label{item::munilp:muvee} $\mu^\vee \in \mathfrak{a}^{-\sigma}$.
    \item \label{item::munilp:mu} $\sigma \mu = -\mu$.
\end{enumerate}
Let us fix any non-zero root vector $X \in \mathfrak{g}_\mu$.
Then one can find $Y \in \mathfrak{g}_{-\mu}$ such that $(\mu^\vee,X,Y)$ is an $\mathfrak{sl}_2$-triple, and $X \in \mathcal{O}^{\C}_{\min,\mathfrak{g}}$ by \cite[Theorem 1.1]{Okuda2015smallest}.
Thus we obtain the equivalence \eqref{item::munilp:nilp} $\Leftrightarrow$ \eqref{item::munilp:hyp} by \cite[Lemma 4.7]{Okuda13}.
Furthermore, by \cite[Proposition 4.6]{Okuda13}, 
we have the equivalence \eqref{item::munilp:hyp}
$\Leftrightarrow$ \eqref{item::munilp:realhyp}.
The equivalence 
\eqref{item::munilp:realhyp}
$\Leftrightarrow$
\eqref{item::munilp:a}
comes from \cite[Lemma 4.9 and Fact 5.1]{Okuda13}.
Since $\mu$ is the highest root of $\Sigma^+(\mathfrak{g},\mathfrak{a})$, 
we have $\mu^\vee \in \mathfrak{a}_+$,
and hence \eqref{item::munilp:a} $\Leftrightarrow$ \eqref{item::munilp:muvee}.
The equivalence \eqref{item::munilp:muvee} $\Leftrightarrow$ \eqref{item::munilp:mu} is easy.
\end{proof}

Let us also give a remark for almost irreducible branching lows as below. 
Let $(G,H)$ be a symmetric pair with a non-compact connected absolutely simple Lie group $G$, and denote by $(\mathfrak{g},\mathfrak{h})$ the corresponding pair of Lie algebras. 
Suppose $G$ has a minimal representation $\Pi_{\min} \in \Irr(G)$ (see \cite[Lemma 5.8]{Kobayashi2023BoundedSymmetric} and \cite{Tamori2019MathZ} for the construction of minimal representations in suitable situations). 
Then $n(\mathfrak{g}_\C) = \DIM(\Pi_{\min})$, and consequently, $n(\mathfrak{g}_\C) = m(\mathfrak{g})$.
Consider the restriction $\Pi_{\min}|_H$ of $\Pi_{\min}$ to $H$. Since $\DIM(\Pi_{\min}) = n(\mathfrak{g}_\C)$, by Fact \ref{fact:bmpKob} \eqref{item:bmpkob1}, the restriction $\Pi_{\min}|_H$ exhibits the bounded multiplicity property (refer to \cite[Theorem 7]{KobayashiMRMR2022}).

In such a situation, as mentioned by Kobayashi \cite[Remark 6.4]{Kobayashi2023BoundedSymmetric}, 
the following theorem holds as an application of 
\cite[Theorem 10]{KobayashiMRMR2022}: 

\begin{theorem}\label{theorem:star}
In the setting above, 
suppose that $(\mathfrak{g},\mathfrak{h})$ can be found  in Table \ref{table:gabssimple}, or equivalently ``$\sigma \mu \neq -\mu$''. 
Then the restriction $\Pi_{\min}|_H$ is almost irreducible, that is, the $H$-module $\Pi_{\min}|_H$ remains irreducible or decomposes into a finite sum of irreducible representations of $H$.
\end{theorem}

\begin{proof}
Let us state Condition $(\star)$ on $(\mathfrak{g},\mathfrak{h})$ as below:
\begin{description}
    \item[Condition ($\star$)] There exists a Riemannian symmetric pair $(\mathfrak{g}_0,\mathfrak{k}_0)$ such that $n(\mathfrak{g}_\C) < m(\mathfrak{g}_0)$ and $(\mathfrak{g}_\C, \mathfrak{h}_\C) := (\mathfrak{g} \otimes_\R \C, \mathfrak{h} \otimes_\R \C)$ is isomorphic to $(\mathfrak{g}_0 \otimes_\R \C,\mathfrak{k}_0 \otimes_\R \C)$.
\end{description}
Kobayashi \cite[Theorem 10]{KobayashiMRMR2022} proved that under the condition $(\star)$, 
the restriction $\Pi_{\min}|_H$ is almost irreducible.

We note that $(\mathfrak{g}^d,\mathfrak{k}(\mathfrak{g}^d))$ is the unique Riemannian symmetric pair whose complexification is isomorphic to $(\mathfrak{g}_\C,\mathfrak{h}_\C)$. 
Thus Condition $(\star)$ is nothing but $n(\mathfrak{g}_\C) < m(\mathfrak{g}^d)$. 
Thus the claim of Theorem \ref{theorem:star} follows directly from \cite[Theorem 10]{KobayashiMRMR2022}, Theorem \ref{theorem:main2} and Proposition  \ref{proposition:sigmamucondition_nilp_condition}.\end{proof}

\begin{remark}
See also \cite{Kobayashi2011Zuckerman60} for other triples $(G,H,\pi)$ with almost irreducible branching laws. 
\end{remark}

\section{Proper $SL_2(\R)$-actions associated to real minimal nilpotent orbits}\label{section:properSL2}

We also give another application of our results in Section \ref{section:classification} for proper $SL_2(\R)$-actions on symmetric spaces (cf.~\cite{Okuda13}). 

Let $(\mathfrak{g},\mathfrak{h})$ be 
a symmetric pair with a simple Lie algebra $\mathfrak{g}$.
We denote by $(\mathfrak{g},\mathfrak{h}^a)$ the associated pair of $(\mathfrak{g},\mathfrak{h})$ in the sense of \cite{OshimaSekiguchi84}.
Note that $\mathfrak{g} \cap \mathfrak{g}^d = \mathfrak{h}^a$.
Let $(G,H^a)$ be a symmetric pair of Lie groups whose symmetric pair of Lie algebras is isomorphic to $(\mathfrak{g},\mathfrak{h}^a)$.
We also suppose that $G$ is linear and connected.

For each Lie algebra homomorphism, 
$\rho : \mathfrak{sl}_2(\R) \rightarrow \mathfrak{g}$, 
we denote by $\mathcal{O}^{\R}_{\rho,\mathrm{nilp}}$ and $\mathcal{O}^{\C}_{\rho,\mathrm{nilp}}$ for the real nilpotent adjoint orbit and the complex nilpotent adjoint orbit through the nilpotent element
\[
\rho\left(\begin{pmatrix} 0 & 1 \\ 0 & 0 \end{pmatrix}\right) \in \mathfrak{g}
\]
in $\mathfrak{g}$ and in $\mathfrak{g}_\C := \mathfrak{g} \otimes_\R \C$, respectively.
Note that by the Jacobson--Morozov theorem and Kostant's theorem for real nilpotent orbits in $\mathfrak{g}$ (see \cite[Chapter 9]{CollingwoodMcGovern} for the details), 
the correspondence $\rho \mapsto \mathcal{O}^{\R}_{\rho,\mathrm{nilp}}$ defines a bijection from the set of 
all $\Int(\mathfrak{g})$-conjugacy classes of $\rho: \mathfrak{sl}_2(\R)  \rightarrow \mathfrak{g}$
onto the set of all real nilpotent orbits in $\mathfrak{g}$, where $\Int(\mathfrak{g})$ denotes the group of all inner-automorphisms on $\mathfrak{g}$.

Let us use the same symbol $\rho$ for the Lie group homomorphism $\rho : SL_2(\R) \rightarrow G$ corresponding to $\rho : \mathfrak{sl}_2(\R) \rightarrow \mathfrak{g}$ (see \cite[Lemma 5.4]{Okuda13}).
Then by \cite[Theorem 10.1]{Okuda13}, which is a corollary to the properness criterion by T.~Kobayashi \cite[Theorem 4.1]{Kobayashi89}, 
the $SL_2(\R)$-action on the symmetric space $G/{H^a}$ via $\rho : SL_2(\R) \rightarrow G$ is proper if and only if $\mathcal{O}^{\C}_{\rho,\mathrm{nilp}} \cap \mathfrak{g}^d = \emptyset$.  

We shall consider a Lie algebra homomorphism $\rho_{\min} : \mathfrak{sl}_2(\R) \rightarrow \mathfrak{g}$ such that $\mathcal{O}^{\R}_{\rho_{\min},\mathrm{nilp}}$ is one of the real minimal nilpotent orbit in $\mathfrak{g}$.
Then $\mathcal{O}^{\C}_{\rho_{\min},\mathrm{nilp}} = \mathcal{O}^{\C}_{\min,\mathfrak{g}}$ and 
the hyperbolic element
\[
\rho_{\min} \left(\begin{pmatrix} 1 & 0 \\ 0 & -1 \end{pmatrix}\right) \in \mathfrak{g}
\]
is conjugate to the coroot $\mu^\vee \in \mathfrak{a}$ of the highest element $\mu$ of $\Sigma^+(\mathfrak{g},\mathfrak{a})$ defined in Section \ref{section:BMB}.

Then by Theorems \ref{theorem:main2} and \ref{theorem:main3}, we have the theorem below:

\begin{theorem}\label{theorem:SL2properGHa}
The following three conditions on $(\mathfrak{g},\mathfrak{h})$ are equivalent: 
\begin{enumerate}
    \item The $SL_2(\R)$-action on $G/{H^a}$ via $\rho_{\min}$ is proper.
    \item $\mathcal{O}^{\C}_{\min,\mathfrak{g}} \cap \mathfrak{g}^d = \emptyset$.
    \item The symmetric pair $(\mathfrak{g},\mathfrak{h})$ is isomorphic to one of the symmetric pairs in Tables \ref{table:gabssimple} or \ref{table:gcomplex}.
\end{enumerate}
\end{theorem}

\section*{Acknowledgements.}
The author would like to give heartfelt thanks to Toshiyuki Kobayashi for his encouragement to write this paper. 
I am also indebted to Koichi Tojo and Ryusei Tsukada for many helpful comments.
I also would like to express my gratitude to the anonymous referees for their meticulous review of the manuscript.
The author is supported by JSPS Grants-in-Aid for Scientific Research JP20K03589, JP20K14310, JP22H01124, and JP24K06714.


\begin{thebibliography}{10}

\bibitem{Araki62}
Sh\^{o}r\^{o} Araki, \emph{On root systems and an infinitesimal classification of irreducible symmetric spaces}, J. Math. Osaka City Univ. \textbf{13} (1962), 1--34. \MR{153782}

\bibitem{Berger_classification}
Marcel Berger, \emph{Les espaces sym\'{e}triques noncompacts}, Ann. Sci. \'{E}cole Norm. Sup. (3) \textbf{74} (1957), 85--177. \MR{0104763}

\bibitem{Brylinski1998realminimalnilp}
Ranee Brylinski, \emph{Geometric quantization of real minimal nilpotent orbits}, vol.~9, 1998, Symplectic geometry, pp.~5--58. \MR{1636300}

\bibitem{CollingwoodMcGovern}
David~H. Collingwood and William~M. McGovern, \emph{Nilpotent orbits in semisimple {L}ie algebras}, Van Nostrand Reinhold Mathematics Series, Van Nostrand Reinhold Co., New York, 1993. \MR{1251060}

\bibitem{Helgason2001bookDLSS}
Sigurdur Helgason, \emph{Differential geometry, {L}ie groups, and symmetric spaces}, Graduate Studies in Mathematics, vol.~34, American Mathematical Society, Providence, RI, 2001, Corrected reprint of the 1978 original. \MR{1834454}

\bibitem{Kobayashi89}
Toshiyuki Kobayashi, \emph{Proper action on a homogeneous space of reductive type}, Math. Ann. \textbf{285} (1989), no.~2, 249--263. \MR{1016093}

\bibitem{Kobayashi2011Zuckerman60}
\bysame, \emph{Branching problems of {Z}uckerman derived functor modules}, Representation theory and mathematical physics, Contemp. Math., vol. 557, Amer. Math. Soc., Providence, RI, 2011, pp.~23--40. \MR{2848919}

\bibitem{KobayashiMRMR2022}
\bysame, \emph{Multiplicity in restricting minimal representations}, Lie theory and its applications in physics, Springer Proc. Math. Stat., vol. 396, Springer, Singapore, [2022] \copyright 2022, pp.~3--20. \MR{4607949}

\bibitem{Kobayashi2023BoundedSymmetric}
\bysame, \emph{Bounded multiplicity branching for symmetric pairs}, J. Lie Theory \textbf{33} (2023), no.~1, 305--328. \MR{4567759}

\bibitem{KobayashiOshimaY2015symmpair}
Toshiyuki Kobayashi and Yoshiki Oshima, \emph{Classification of symmetric pairs with discretely decomposable restrictions of {$(\mathfrak{g},K)$}-modules}, J. Reine Angew. Math. \textbf{703} (2015), 201--223. \MR{3353547}

\bibitem{Okuda13}
Takayuki Okuda, \emph{Classification of semisimple symmetric spaces with proper {$SL(2,\Bbb R)$}-actions}, J. Differential Geom. \textbf{94} (2013), no.~2, 301--342. \MR{3080484}

\bibitem{Okuda2015smallest}
\bysame, \emph{Smallest complex nilpotent orbits with real points}, J. Lie Theory \textbf{25} (2015), no.~2, 507--533. \MR{3346070}

\bibitem{OshimaSekiguchi84}
Toshio \={O}shima and Jir\={o} Sekiguchi, \emph{The restricted root system of a semisimple symmetric pair}, Group representations and systems of differential equations ({T}okyo, 1982), Adv. Stud. Pure Math., vol.~4, North-Holland, Amsterdam, 1984, pp.~433--497. \MR{810638}

\bibitem{Tamori2019MathZ}
Hiroyoshi Tamori, \emph{Classification of minimal representations of real simple {L}ie groups}, Math. Z. \textbf{292} (2019), no.~1-2, 387--402. \MR{3968907}

\end{thebibliography}

\providecommand{\bysame}{\leavevmode\hbox to3em{\hrulefill}\thinspace}
\providecommand{\MR}{\relax\ifhmode\unskip\space\fi MR }
\providecommand{\MRhref}[2]{%
  \href{http://www.ams.org/mathscinet-getitem?mr=#1}{#2}
}
\providecommand{\href}[2]{#2}

\end{document}